\documentclass[12pt, reqno]{amsart}
\usepackage[T1]{fontenc}
\usepackage[utf8]{inputenc}
\usepackage{amsmath,amsthm,amssymb,amsfonts,csquotes}
\usepackage{fullpage}
\usepackage{tikz,float}
\usepackage[pdftex, urlcolor=blue,colorlinks, linkcolor=red,citecolor=blue]{hyperref}
\usepackage{cleveref}
\usepackage{subcaption}
\usepackage{comment}
\usepackage[english]{babel}
\usepackage{xcolor}

\usepackage{todonotes}

\newtheorem{theorem}{Theorem}[section]
\newtheorem{lemma}[theorem]{Lemma}
\newtheorem{claim}[theorem]{Claim}
\newtheorem{corollary}[theorem]{Corollary}
\newtheorem{conj}[theorem]{Conjecture}

\newtheorem{proposition}[theorem]{Proposition}

\newtheorem{definition}[theorem]{Definition}

\newtheorem{remark}[theorem]{Remark}
\newtheorem{example}[theorem]{Example}
\numberwithin{equation}{section}
\newtheorem*{thm}{Theorem}

\newcommand{\ind}{{\rm Ind}}
\newcommand{\con}{{\rm Con}}
\newcommand{\lk}{{\rm lk}}
\newcommand{\del}{{\rm del}}
\newcommand{\Indr}[1]{{\rm Ind}_{#1}}

\def\K{\mathcal{K}}
\def\H{\mathcal{H}}

\title{Chordal graphs, higher independence and vertex decomposable complexes}
\author{Fred M. Abdelmalek}
\address{University of Toronto, Canada}
\email{fred.abdelmalek@mail.utoronto.ca}
\author{Priyavrat Deshpande}
\address{Chennai Mathematical Institute, India}
\email{pdeshpande@cmi.ac.in}
\author{Shuchita Goyal}
\address{Indian Institute of Technology Kanpur, India}
\email{shuckriya.goyal@gmail.com}
\author{Amit Roy}
\address{School of Mathematical Sciences, National Institute of Science Education and Research, Bhubaneswar, 752050, India $\&$ \newline \hspace*{0.43cm}Homi Bhabha National Institute, Training School Complex, Anushaktinagar, Mumbai 400094, India}
\email{amitiisermohali493@gmail.com}
\author{Anurag Singh}
\address{Indian Institute of Technology Bhilai, India}
\email{anurags@iitbhilai.ac.in}

\begin{document}
\begin{abstract}
Given a finite simple undirected graph $G$ there is a simplicial complex $\mathrm{Ind}(G)$, called the independence complex, whose faces correspond to the independent sets of $G$. 
This is a well studied concept because it provides a fertile ground for interactions between commutative algebra, graph theory and algebraic topology. 
In this article we consider a generalization of independence complex. 
Given $r\geq 1$, a subset of the vertex set is called $r$-independent if the connected components of the induced subgraph have cardinality at most $r$. 
The collection of all $r$-independent subsets of $G$ form a simplicial complex called the $r$-independence complex and is denoted by $\mathrm{Ind}_r(G)$. 
It is known that when $G$ is a chordal graph the complex $\mathrm{Ind}_r(G)$ has the homotopy type of a wedge of spheres. 
Hence it is natural to ask which of these complexes are shellable or even vertex decomposable. We prove, using Woodroofe's chordal hypergraph notion, that these complexes are always shellable when the underlying chordal graph is a tree.
Using the notion of vertex splittable ideals we show that for caterpillar graphs the associated $r$-independence complex is vertex decomposable for all values of $r$. 
Further, for any $r\geq 2$ we construct chordal graphs on $2r+2$ vertices such that their $r$-independence complexes are not sequentially Cohen-Macaulay.

\end{abstract}
\maketitle

\section{Introduction}

Let $G$ be a (simple) graph with the vertex set $V(G)$ and the edge set $E(G)$. For a subset $U$ of $V(G)$, the induced subgraph $G[U]$ is the subgraph of $G$ with vertices $V(G[U]) = U$ and edges $E(G[U]) = \{(a, b) \in E(G) \ | \ a, b \in U\}$. For $r\geq 1$, a subset $A\subseteq V(G)$ is called {\it r -independent} if connected
components of the induced subgraph $G[A]$ have cardinality (number of vertices) at most $r$.

The {\it independence complex} of a graph $G$, denoted $\ind(G)$, is a simplicial complex whose simplices are all $1$-independent subsets of $G$. 
The topological study of the independence complexes of graphs has received a lot of attention in last two decades. 
For instance, in Babson and Kozlov’s proof of Lov{\'a}sz's conjecture \cite{BK07} regarding odd
cycles and graph homomorphism complexes, the independence complexes of cycle graphs played an important role. Also, Meshulam \cite{Meshulam03} gave a connection between the domination number of a graph $G$ and homological connectivity of $\ind(G)$.
Properties of independence complexes have also been used to study the Tverberg graphs \cite{Eng11} and the independent  system of representatives \cite{abz07}.

The independence complexes have also been studied in combinatorial commutative algebra in connection with the edge ideals of graphs. More precisely, for a graph $G$ on $n$ number of vertices, the {\it Stanley-Reisner ideal} of $\operatorname{Ind}(G)$ is the edge ideal $I(G)$ in the polynomial ring with $n$ variables over a field $\mathbb K$. Various algebraic and combinatorial properties of $I(G)$ have been studied recently. 
For a chordal graph $G$, Herzog, Hibi and Zheng \cite{HHZ06} proved that $\ind(G)$ is Cohen-Macaulay if and only if $I(G)$ is unmixed. In this context Francisco and Van Tuyl \cite{FV07} showed that the ideal $I(G)$ is sequentially Cohen-Macaulay whenever $G$ is chordal.
Later, Engstr\"om and Dochtermann \cite{DE09}, and Woodroofe \cite{Woodroofe09}  generalized this result by showing that for a chordal graph $G$, $\operatorname{Ind}(G)$ is vertex decomposable and hence $I(G)$ is sequentially Cohen-Macaulay.

The aim of the article is to study the simplicial complex that arises in the context of $r$-independent sets and naturally generalizes independence complex. 
We begin by a formal definition. 

\begin{definition}
\normalfont The {\it $r$-independence complex} of a graph $G$, denoted $\ind_r(G)$, is the simplicial complex whose vertices are the vertices of $G$ and simplices are all $r$-independent subsets of $G$.
\end{definition}

 See \Cref{fig:example of ind complex} for an example. The $1$-independence complex of $G$ consists of $3$ maximal simplices, namely $\{v_2,v_3,v_4\}$, $\{v_3,v_4,v_5\}$ and $\{v_1,v_5\}$. 
The complex $\ind_2(G)$ consists of $4$ maximal simplices, namely $\{v_1,v_2\},  \{v_1,v_3,v_5\}, \{v_1,v_4,v_5\}$ and $\{v_2,v_3,v_4,v_5\}$.

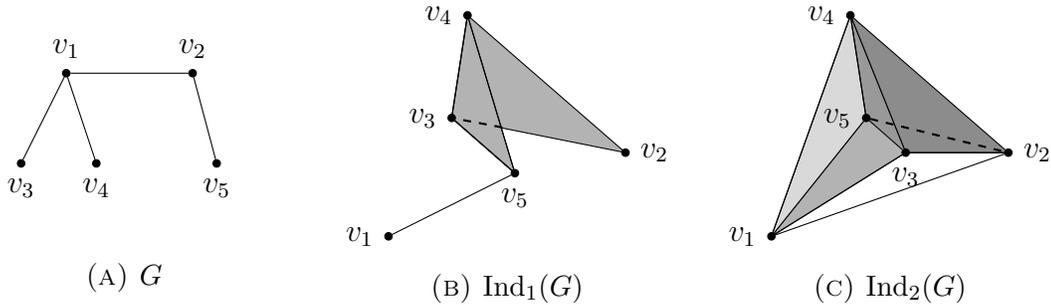
\begin{figure}[!ht]
	\begin{subfigure}[]{0.30 \textwidth}
		\centering
		\vspace{0.3cm}
		\begin{tikzpicture}
 [scale=0.4, vertices/.style={draw, fill=black, circle, inner sep=1.0pt}]
        \node[vertices, label=above: {$v_1$}] (v1) at (0,3)  {};
		\node[vertices, label=above: {$v_2$}] (v2) at (4.2,3)  {};
		\node[vertices, label=below:
		{$v_3$}] (l11) at (-1.5,0)  {};
		\node[vertices, label=below:
		{$v_4$}] (l12) at (1,0)  {};
		\node[vertices, label=below:
		{$v_5$}] (l21) at (5,0)  {};
		
\foreach \to/\from in {v1/v2}
\path (v1) edge node[pos=0.5,below] {} (v2);
\path (v1) edge node[pos=0.5,left] {} (l11);
\path (v1) edge node[pos=0.5,left] {} (l12);
\path (v2) edge node[pos=0.5,left] {} (l21);
\end{tikzpicture}
\vspace{0.5cm}\caption{$G$}\label{fig:G221}
	\end{subfigure}
	\begin{subfigure}[]{0.30 \textwidth}
		\centering
	\begin{tikzpicture}
 [scale=0.21, vertices/.style={draw, fill=black, circle, inner sep=1.0pt}]

\filldraw[fill=gray!60] (15,5.3)--(5,14)--(4,7.5)--cycle;
\filldraw[fill=gray!60] (8,4)--(4,7.5)--(5,14)--cycle;
\node[vertices, label=left:{$v_1$}] (a) at (0,0) {};
\node[vertices, label=left:{$v_4$}] (b) at (5,14) {};
\node[vertices, label=below:{$v_5$}] (c) at (8,4) {};
\node[vertices, label=right:{$v_2$}] (d) at (15,5.3) {};
\node[vertices, label=left:{$v_3$}] (e) at (4,7.5) {};
\foreach \to/\from in {a/c,e/c,e/e,c/e,c/b,b/e}
\draw [-] (\to)--(\from);
\draw [thick,dashed] (7.2,6.9)--(e);
\end{tikzpicture}\caption{$\ind_1(G)$}
	\end{subfigure}
	\begin{subfigure}[]{0.30 \textwidth}
		\centering
	\begin{tikzpicture}
 [scale=0.21, vertices/.style={draw, fill=black, circle, inner sep=1.0pt}]

\fill[fill=gray!80] (15,5.3)--(6,7.5)--(8.5,5.3);
\filldraw[fill=gray!80] (15,5.3)--(6,7.5)--(5,14);
\filldraw[fill=gray!90] (15,5.3)--(8.5,5.3)--(5,14);
\fill[fill=gray!80] (5,14)--(6,7.5)--(8.5,5.3);
\filldraw[fill=gray!60] (8.5,5.3)--(0,0)--(6,7.5);
\filldraw[fill=gray!30] (6,7.5)--(5,14)--(0,0);

\node[vertices, label=left:{$v_1$}] (a) at (0,0) {};
\node[vertices, label=left:{$v_4$}] (c) at (5,14) {};
\node[vertices, label=below:{$v_3$}] (f) at (8.5,5.3) {};
\node[vertices, label=right:{$v_2$}] (e) at (15,5.3) {};
\node[vertices, label=left:{$v_5$}] (b) at (6,7.5) {};
\foreach \to/\from in {a/b,a/f,a/c,b/c,f/c,b/f,e/f,e/c,a/e}
\draw [-] (\to)--(\from);
\draw [thick,dashed] (b)--(e);
\end{tikzpicture}\caption{$\ind_2(G)$}
	\end{subfigure}
	\caption{Example of higher independence complexes} \label{fig:example of ind complex}
\end{figure}

These complexes have appeared in the work of Szab{\'o} and Tardos \cite{ST06} and also in the work of Paolini and Salvetti \cite{PS18}. 
The authors, in \cite{DS21}, proved that the $r$-independence complexes are homotopy equivalent to a wedge of spheres when the graph in question is either a cycle graph or a perfect $m$-ary tree.  
Extending Meshulam's result \cite{Meshulam03}, it was proved in \cite{DSS20} that the (homological) connectivity of $r$-independence complexes of graphs give an upper bound for the distance $r$-domination number of graphs. 
In the same paper the authors also proved that the $r$-independence complexes of chordal graphs are homotopy equivalent to a wedge of spheres for each $r\geq 1$. 

The next obvious direction in the study of theses complexes is to determine for which graph classes and for what values of $r$ the $r$-independence complex is sequentially Cohen-Macaulay (\Cref{def:scm})?
Due to the seminal work of Reisner \cite[Theorem 6.3.12]{RV1} we know that the topological notion of sequentially Cohen-Macaulay is equivalent to the classical, commutative algebra notion of sequentially Cohen-Macaulay rings. Recall that the Stanley-Reisner ring of a simplicial complex $\K$ is the quotient of the polynomial ring in $|V(\K)|$ variables by the square-free monomial ideal generated by (minimal) non-faces of $\K$.  

The notion of shellable complexes is stronger than sequentially Cohen-Macaulay complexes. 
Even stronger notion is that of a \emph{vertex decomposable} complex, it implies shellability and either of these conditions mean that the simplicial complex under consideration has the homotopy type of a wedge of spheres. 
As stated before, for several graph classes their $r$-independence complexes have the homotopy type of a wedge of spheres.
Moreover, if the graph $G$ has $r+1$ vertices then $\ind_r(G)$ is the boundary of a simplex, hence vertex decomposable. 

Recall that if $G$ is a chordal graph then $\ind_1(G)$ is vertex decomposable \cite{DE09, Woodroofe09}. 
It is therefore natural to ask whether for every value of $r$, is the $r$-independence complex of a chordal graph vertex decomposable?
The answer to this question is no; as we show in \cref{sec:concluding}, there are some chordal graphs such that one of their $r$-independence complexes is not sequentially Cohen-Macaulay.

Using Woodroofe's notion of chordal hypergraphs, we show that the $r$-independence complexes associated to trees are always shellable. 

\begin{thm}[\Cref{cor:higherindshellable}]
For any tree $T$ and $r \geq 1$, the complex $\ind_r(T)$ is shellable.
\end{thm}

Recall that, vertex decomposability is defined recursively in terms of link and deletion of a vertex (see \Cref{vddef}). 
In case of higher independence complexes the link of any vertex is usually quite complicated and lacks in obvious pattern. Hence, with the aid of tools from commutative algebra we prove the following.

\begin{thm}[Theorem \ref{main theorem *}]
For $r\geq 1,$ the $r$-independence complexes of caterpillar graphs are vertex decomposable. 
\end{thm}

Finally, in \Cref{sec:concluding}, we give examples of chordal graphs for which the behaviour of higher independence complex differs from that of $1$-independence complex.  For instance, when $r\geq 2$, we construct chordal graphs $H_r$ and $G_r$, both on $2r+2$ vertices, such that $\mathrm{Ind}_{r+1}(H_{r})$ is homotopy equivalent to a wedge of $r$-dimensional spheres and $\ind_{r}(G_r)$ is contractible but neither of them is sequentially Cohen-Macaulay.

\section{Shellable higher independence complexes}\label{sec:shellable}

 The elements of a simplicial complex $\K$ are called {\it faces} (or {\itshape simplices}) of $\K$.  If $F \in \K$ and $|F |=k+1$, then $F$ is said to be {\it $k$-dimensional}. The \emph{dimension} of $\K$, denoted dim$(\K)$, is the maximum of the dimensions of its faces. The set of $0$-dimensional simplices of $\K$ is denoted by $V(\K)$, and its elements are called {\it vertices} of $\K$. The maximal faces with respect to inclusion are called the {\it facets} of $\K$. The complex $\K$ is called {\it pure} if all its facets are of the same dimension. A {\it subcomplex} of a simplicial complex $\K$ is a simplicial complex whose simplices are contained in $\K$. For a set $F$, let $\Delta^F$ denotes a simplicial complex whose faces are all subsets of $F$.

    \begin{definition}\label{def:shellability}
        \normalfont A simplicial complex $\K$ is called \emph{shellable} if the facets of $\K$ can be arranged in linear order $F_1, F_2, \dots ,F_t$ in such a way that the subcomplex $\big(\bigcup\limits_{1\leq j <r}\Delta^{F_j}\big) \cap \Delta^{F_r}$ is pure and
        $(\text{dim}(\Delta^{F_k}) -1)$-dimensional for all $k = 2,\dots,t$. Such an ordering of facets is called a
        {\it shelling order} of $\K$.
\end{definition}

The main aim of this section is to show that the $r$-independence complexes of chordal graphs are shellable. To do so, we associate a special hypergraph $\con_r(G)$ to a graph $G$ such that the independence complex of $\con_r(G)$ is same as the $r$-independence complex of $G$.  We then use Woodroofe's \cite{Woodroofe11} notion of chordality of a hypergraph to prove our result. We start by introducing the required terminology of hypergraphs and their independence complexes.

	A {\it hypergraph} $\H$ is a pair $(V(\H), E(\H))$ consisting of a set $V(\H)$ along with a subset $E(\H)$ of $2^{V(\H)}$. A hypergraph is said to be {\it simple} if for any $e \in E(\H)$, there is no $e' \in E(\H)$ satisfying $e' \subsetneq e$.
	A subset $S$ of $V(\H)$ is called {\it independent} if for no $e \in E(\H)$, $e \subset S$. The {\it independence complex of a hypergraph} $\H$, denoted by $\ind(\H)$, is a simplicial complex whose simplices are formed by independent subsets of $V(\H)$.

	Let $\H$ be a hypergraph and $v \in V(\H)$. The {\it deletion of vertex } $v $ is a hypergraph, denoted by $\H \setminus v$, with $V(\H \setminus v) = V(\H) \setminus \{v\}$ and $E(\H \setminus v) = \{e \in E(\H): v  \notin e\}$. 
	The {\it contraction of a vertex} $v$ is a hypergraph, denoted by $\H/ v$, with $V(\H / v) = V(\H) \setminus \{v\}$ and $E(\H / v) = \{e \setminus \{v\}: e \in E(\H)\}$. 
	A {\it minor} of a hypergraph $\H$ is a hypergraph obtained from $\H$ via a sequence of contractions and deletions.
	We note that the contraction of a vertex in a  simple hypergraph need not yield a simple hypergraph, and hence a minor of a simple hypergraph may not be simple. Since we deal with independence complex, it suffices to consider only the simple hypergraphs. In lieu of this, whenever the minor of a simple hypergraph is not simple, we consider the underlying simple hypergraph of the given minor.
    
    \begin{definition}\rm{\cite[Definition 4.2]{Woodroofe11}}
    {\normalfont
        In a hypergraph $\H$, a vertex $v \in V(\H)$ is called a {\it simplicial vertex} if for any two edges $e_1, e_2 \in E(\H)$ that contains $v$, there is an $e_3 \in E(\H)$ such that $e_3 \subset (e_1 \cup e_2) \setminus \{v\}$.
    }
    \end{definition}
    
    \begin{definition}\rm{\cite[Definition 4.3]{Woodroofe11}}\label{def:simplicial vertex in chordal graph}
    {\normalfont
        A hypergraph $\H$ is {\it chordal} if every minor of $\H$ has a simplicial vertex.}
    \end{definition}
    
    We note that for a graph $G$, its edge set $E(G)$ can be regarded as a subset of $ 2^{V(G)}$ such that for any $e \in E(G), |e| = 2$. Observe that any minor of a graph is an induced subgraph of the graph and vice-versa. So chordal graphs are also examples of chordal hypergraphs. 
    Therefore the above definition of chordal hypergraphs generalizes the notion of chordality of graphs.

    \begin{theorem}{\rm{\cite[Theorem 1.1]{Woodroofe11}}}\label{thm:woodroofe shellable}
	    If $\H$ is a chordal hypergraph, then the independence complex $\ind(\H)$ is shellable.
	\end{theorem}

   \begin{definition}\label{def:conrG}
        Let $G$ be a graph and $r \ge 1$. We define a hypergraph $\con_r(G)$ associated to $G$ with $V(\con_r(G)) = V(G)$ and the edge set of $\con_r(G)$ is the collection of all subsets $S\subseteq V(G)$ such that $|S|=r+1$ and the induced subgraph $G[S]$ is connected.
   \end{definition} 

	\begin{proposition}
 		For any graph $G$, $\ind (\con_{r}(G)) = \ind_r(G)$.
	\end{proposition}

    \begin{proof}
            
        By definition of the independence complex of hypergraphs, $\sigma \in \ind(\con_r(G))$ if and only if $\sigma \cap E(\con_r(G))=\emptyset$, {\it i.e.}, the induced subgraph $G[\sigma]$ has no connected component of cardinality $r+1$. Therefore by definition of $\con_r(G)$ and $\ind_r(G)$, $\sigma \in \ind(\con_r(G))$ if and only if $\sigma \in \ind_r(G)$.
    \end{proof}
    
We now prove the following result which would imply the main result of this section.
	\begin{theorem}\label{thm:mainshellable}
For any tree $G$, the hypergraph $\con_r(G)$ is chordal.
	\end{theorem}
	
	\begin{proof}
		Let $V_1$ be the set of all the simplicial vertices  of $G$. For $i \ge 2$, let $V_i$ be the set of all the simplicial vertices  of $ G \setminus \bigcup_{j=1}^{i-1} V_j$. 
		By  \Cref{def:simplicial vertex in chordal graph}, each of these $V_i$'s is non-empty. We show that any minor hypergraph $H$ of the hypergraph $\con_r(G)$ has a simplicial vertex. Let $\ell = \min \{i : V(H) \cap V_i \ne \emptyset\}$. Choose a vertex $\tilde{v} \in V_{\ell} \cap V(H)$. We show that $\tilde{v}$ is simplicial in $H$.
		
		Let $e_1,e_2 \in E(H)$ be two distinct edges containing $v$. Let $f_1,f_2 \in E(\con_r(G))$ be such that $f_1 \cap V(H) = e_1$, $f_2 \cap V(H) = e_2$ and all the vertices  of $f_1 \setminus e_1$ and $ f_2 \setminus e_2$ have been contracted in $\con_r(G)$ to reach $H$. Choose $x_1 \in e_1 \setminus e_2$ and $x_2 \in e_2 \setminus e_1$ such that,
		\begin{equation*}
		    \begin{split}
		    d_{G[f_1]}(x_1,\tilde{v})  & \le d_{G[f_1]}(x,\tilde{v}) ~ \mathrm{~ for~all~} x \in f_1,
		 \text{ and } \\
		 d_{G[f_2]}(x_2,\tilde{v}) & \le d_{G[f_2]}(y,\tilde{v}) ~ \mathrm{~ for~all~} y \in f_2,
		 \end{split}
		\end{equation*}
		where $d_G(u,v)$ denotes the distance\footnote{In a graph $G$, for $u,v \in V(G)$, the {\it distance} between $u$ and $v$, denoted $d_G(u,v)$, is the length of the shortest path in $G$ with end points $u$ and $v$. When the context is clear, we drop the subscript $G$ from $d_G(u,v)$. 
    }  between vertices $u$ and $v$ in $G$. 
    
		Let $g_1 = \tilde{v} y_1 \cdots y_{p}  x_1$ and $g_2 = \tilde{v} z_1 \cdots z_{q} x_2$ be the minimum length paths in $G[f_1]$ and $G[f_2]$ respectively. Since
		$\tilde{v} \in g_1 \cap g_2$, the subgraph $g_1 \cup g_2$ is connected. By definition of $x_1$ and $g_1$,  $\{y_1, y_2, \dots, y_p\} \cap e_1 = \emptyset$. Similarly, $\{z_1, z_2, \dots, z_q\} \cap e_2 = \emptyset$. Therefore, $V(g_1 \cup g_2 ) \cap V(H) = \{\tilde{v}, x_1,x_2\}$. 

\begin{claim}\label{claim:shellable}
        $V(g_1 \cup g_2) \bigcap (\displaystyle \cup_{i=1}^{\ell-1} V_i) = \emptyset$.
\end{claim}		
		
	\begin{proof}[Proof of \Cref{claim:shellable}]
	 We prove it by contradiction. Suppose $V(g_1 \cup g_2) \bigcap (\displaystyle \cup_{i=1}^{\ell-1} V_i) \ne \emptyset$. Without loss of generality, assume that $V(g_1) \cap (\displaystyle \cup_{i=1}^{\ell-1} V_i) \ne \emptyset$. Since $\tilde{v}, x_1 \notin \cup_{i=1}^{\ell-1} V_i$, $V(g_1) \cap (\displaystyle \cup_{i=1}^{\ell-1} V_i) \ne \emptyset$ implies that $x_1 \notin N_G(\tilde{v})$. Let $t = \min \{i: V(g_1) \cap V_i \ne \emptyset\}$. Clearly $1 \le t < \ell$. Let $y_j$ be the first vertex in the path $g_1 = \tilde{v} y_1 \cdots y_{p}  x_1$ such that $y_j \in V_t$. 
		Since $y_j$ is a simplicial vertex in $G[\cup_{i \ge t} V_i]$, $y_{j-1}$ is adjacent to $y_{j+1}$. Thus $g_1' = \tilde{v} y_1 \cdots y_{j-1} y_{j+1} \dots y_{p}  x_1$ is a smaller length path than $g_1$ in $G[f_1]$ connecting $\tilde{v}$ and $x_1$, which is a contradiction. If $y_j \in \{y_1, y_p\}$, then similar arguments lead us to a smaller length path connecting $\tilde{v}$ and $x_1$, which again is a contradiction.
	\end{proof}

		Since $\tilde{v} \in V_{\ell}$ is a simplicial vertex in $G[\cup_{i \ge \ell}V_i]$, either $y_1 = z_1$ or $y_1$ is adjacent to $z_1$. Thus $g=G[V(g_1 \cup g_2) - \{\tilde{v}\}]$ is a connected subgraph of $G$. We now find an edge $e_3\in E(H)$ such that $e_3 \subseteq (e_1 \cup e_2)\setminus \{\tilde{v}\}$.
		
		If $|V(g)| \geq r+1$, then take a connected subgraph $f$ of $g$ of cardinality $r+1$ containing $x_1$ and let  $e_3$ to be $V(f)\cap V(H)$. If $|V(g)| \leq r$, then extend $g$ to an $(r+1)$-connected subgraph $f$ of $G[(f_1 \cup f_2) - \{v_1\}]$ and let $e_3$ to be $V(f)\cap V(H)$. 
	\end{proof}
	
	For a better understanding of the proof of \Cref{thm:mainshellable}, we give an example here.

	\begin{figure}[]
		\centering
		\begin{tikzpicture}
			[scale=0.6, vertices/.style={draw, fill=black, circle, inner
				sep=1pt}, NoVert/.style={draw, circle, inner
				sep=1.2pt}]
			
			\draw[line width=7pt,orange!65,double distance=1pt] (12,11) -- (16,15);
			\draw[line width=7pt,orange!65,double distance=1pt] (16,13) -- (16,15);
			\draw[line width=7pt,orange!65,double distance=1pt] (16,13) -- (14,11);
			\node[fill=orange!65, circle, inner
				sep=6pt] () at (16, 15) {};
			\node[fill=orange!65, circle, inner
				sep=5.5pt] () at (16, 13) {};
			\node[fill=orange!65, circle, inner
				sep=5.5pt] () at (12, 11) {};
			\node[fill=orange!65, circle, inner
				sep=5.5pt] () at (14, 11) {};
			\draw[line width=3pt,blue!35,double distance=1pt] (16,13) -- (16,15);
			\draw[line width=3pt,blue!35,double distance=1pt] (16,13) -- (14,11);
			\draw[line width=3pt,blue!35,double distance=1pt] (16,15) -- (20,11);
			\draw[line width=3pt,orange!35,double distance=1pt] (16,15) -- (16,15);
			\node[fill=blue!35, circle, inner
				sep=3pt] () at (16, 15) {};
			\node[fill=blue!35, circle, inner
				sep=2.8pt] () at (16, 13) {};
			\node[fill=blue!35, circle, inner
				sep=2.8pt] () at (20, 11) {};
			\node[fill=blue!35, circle, inner
				sep=2.8pt] () at (14, 11) {};
			\node[circle, inner
				sep=3pt] (f1) at (11, 14.5) {$T[f_1]$};
			\draw [thick,->,dashed] (f1)-- (14.4,14);
			\node[circle, inner
				sep=3pt] (f2) at (20.5, 14.5) {$T[f_2]$};
			\draw [thick,->,dashed] (f2)--(17.4,14);

			\node[NoVert,label=above:{{$v_{5,1}$}}] (0) at (10, 17) {};
			\node[vertices,label=left:{{$v_{4,1}$}}] (1) at (5, 15) {};
			\node[vertices,label=right:{{$v_{4,2}$}}] (2) at (16, 15) {};
			\node[vertices,label=left:{{$v_{3,1}$}}] (3) at (3, 13) {};
			\node[vertices,label=right:{{$v_{3,2}$}}] (4) at (7, 13) {};
			\node[vertices,label=left:{{$v_{3,3}$}}] (5) at (14, 13) {};
			\node[NoVert,label=right:{{$v_{3,4}$}}] (6) at (16, 13) {};
			\node[vertices,label=right:{{$v_{3,5}$}}] (7) at (18, 13) {};				
			\node[vertices,label=above:{{$v_{2,1}$}}] (8) at (1, 11) {};
			\node[vertices,label=above:{{$v_{2,2}$}}] (9) at (5, 11) {};
			\node[NoVert,label=above:{{$v_{2,3}$}}] (10) at (9, 11) {};
			\node[vertices,label=above:{{$v_{2,4}$}}] (11) at (12, 11) {};								
			\node[vertices,label=right:{{$v_{2,5} = \tilde{v}$}}] (12) at (14, 11) {};
			\node[vertices,label=above:{{$v_{2,6}$}}] (13) at (18, 11) {};
			\node[vertices,label=above:{{$v_{2,7}$}}] (14) at (20, 11) {};
			
			\node[NoVert,label=below:{{$v_{1,1}$}}] (15) at (1, 9) {};
			
			\node[NoVert,label=below:{{$v_{1,2}$}}] (16) at (5, 9) {};
			\node[NoVert,label=below:{{$v_{1,3}$}}] (17) at (9, 9) {};
			\node[NoVert,label=below:{{$v_{1,4}$}}] (18) at (12, 9) {};
			\node[NoVert,label=below:{{$v_{1,5}$}}] (19) at (14, 9) {};
			\node[NoVert,label=below:{{$v_{1,6}$}}] (20) at (18, 9) {};
			\node[NoVert,label=below:{{$v_{1,7}$}}] (21) at (20, 9) {};
			
			\foreach \to/\from in
			{ 1/4, 2/5, 2/7, 1/3, 3/8, 3/9, 5/11, 7/14} \draw [-] (\to)--(\from);

			
			\foreach \to/\from in 
 			{0/1, 0/2, 2/6, 6/12, 6/13, 4/10, 10/17, 8/15, 9/16, 11/18, 12/19, 13/20, 14/21} \draw [dotted] (\to)--(\from);

		\end{tikzpicture}
		\caption{Illustration of the proof of \Cref{thm:mainshellable}}	\label{fig:examplecomputation}
	\end{figure}
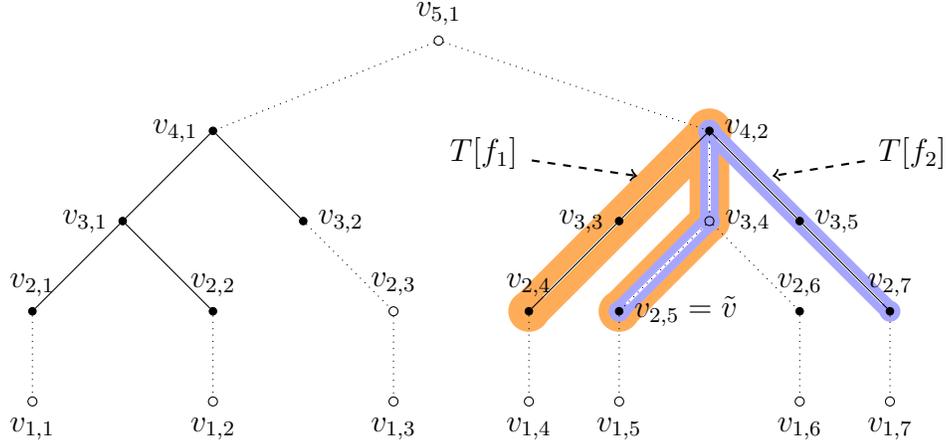

	\begin{example}
	\normalfont
	A minor $H$ of a tree $T$ is depicted in  \Cref{fig:examplecomputation}, where hollow vertices have been contracted. Let $r=4$ and pick $v_{2,5} \in V_2 \cap H$ to be $\tilde{v}$. Let $e_1 = \{\tilde{v}, v_{4,2}, v_{3,3}, v_{2,4}\}$ and $e_2 = \{\tilde{v}, v_{4,2}, v_{3,5}, v_{2,7}\}$, then $x_1 = v_{3,3} \in e_1 \setminus e_2$ and $x_2 = v_{3,5} \in e_2 \setminus e_1$. Take $f_1 = \{\tilde{v}, v_{3,4}, v_{4,2}, v_{3,3}, v_{2,4}\}$ and $f_2 = \{\tilde{v}, v_{3,4}, v_{4,2}, v_{3,5}, v_{2,7}\}$ (as depicted in \Cref{fig:examplecomputation}), then $g_1 = \tilde{v} v_{3,4}v_{4,2} v_{3,3}$ and $g_2 = \tilde{v} v_{3,4}v_{4,2} v_{3,5}$. So $g = T[\{v_{3,4},v_{4,2}, v_{3,3}, v_{3,5}\}]$ is a connected component of $T$ and $|g| = 4 < r+1$. Choosing $f$ to be either $\{v_{3,4},v_{4,2}, v_{3,3}, v_{3,5}, v_{2,4}\}$ or $\{v_{3,4},v_{4,2}, v_{3,3}, v_{3,5}, v_{2,7}\}$ and taking $e_3=f\cap V(H)$ suffices.
	\end{example} 
	
	\begin{corollary}\label{cor:higherindshellable}
		For any tree $T$, $\ind_r(T)$ is shellable.
	\end{corollary}
	
	\begin{proof}
		From \Cref{thm:mainshellable}, the associated hypergraph $\con_{r}(T)$ is chordal. Therefore, the statement follows from  \Cref{thm:woodroofe shellable}.
	\end{proof}

\section{Vertex decomposable complexes}\label{sec:VD}

Let $\K$ be a simplicial complex. For a simplex $F \in \K$, define
\begin{equation*}
    \begin{split}
        \mathrm{lk}_{\K}(F) & := \{F^\prime \in \K : F \cap F^\prime = \emptyset,~ F \cup F^\prime \in \K\}, \\
        \mathrm{del}_{\K}(F) & := \{F^\prime \in \K : F \nsubseteq F^\prime\}.
    \end{split}
\end{equation*}

The simplicial complexes $\mathrm{lk}_{\K}(F)$ and $\mathrm{del}_{\K}(F)$ are called \emph{link} of $F$ and \emph{(face) deletion} of $F$ in $\K$ respectively. For $x\in V(\K)$, to simplify the notations, we denote $\lk_{\K}(\{x\})$ and $\del_{\K}(\{x\})$ by $\lk_{\K}(x)$ and $\del_{\K}(x)$, respectively. 

\begin{definition}\label{vddef}
\normalfont A simplicial complex $\K$ is said to be {\it vertex decomposable} if either $\K$ is a simplex or there exists $x\in V(\K)$ such that
 
 \begin{enumerate}
  \item[$\bullet$] both $\lk_{\K}(x)$ and $\del_{\K}(x)$ are vertex decomposable, and
  \item[$\bullet$] each facet of $\del_{\K}(x)$ is a facet of $\K$.
 \end{enumerate}
The vertex $x$ is called a \textit{shedding vertex} if the second condition above is satisfied.  
\end{definition}

It is easy to observe that when $G$ is a complete graph,  $\ind_r(G)$ is the $(r-1)$-skeleton of a simplex and hence it is vertex decomposable.
The fact that a vertex decomposable simplicial complex is shellable can be easily obtained from Wachs' result \cite[Lemma 6]{wachs99}.

Proving the vertex decomposability of the link of any vertex in higher independence complex is fairly difficult. 
We illustrate this using an example, let us take $r=2$ and $G=P_7$.  

\begin{figure}[ht]
\centering
\begin{subfigure}[]{0.35\textwidth}
    \centering
    \begin{tikzpicture}[scale=.55]
\draw [fill] (0,3) circle [radius=0.1];
\draw [fill] (2,3) circle [radius=0.1];
\draw [fill] (4,3) circle [radius=0.1];
\draw [fill] (6,3) circle [radius=0.1];
\draw [fill] (8,3) circle [radius=0.1];
\draw [fill] (10,3) circle [radius=0.1];
\draw [fill] (12,3) circle [radius=0.1];
\node at (0,3.5) {$1$};
\node at (2,3.5) {$2$};
\node at (4,3.5) {$3$};
\node at (6,3.5) {$4$};
\node at (8,3.5) {$5$};
\node at (10,3.5) {$6$};
\node at (12,3.5) {$7$};
\draw (0,3)--(2,3)--(2,3)--(4,3)--(4,3)--(6,3)--(6,3)--(8,3)--(8,3)--(10,3)--(10,3)--(12,3);
\end{tikzpicture}
    \caption{$P_7$}\label{fig:p7}
\end{subfigure}
\hfill
\begin{subfigure}[]{0.35\textwidth}
    \centering
    \begin{tikzpicture}[scale=.55]
\draw [fill] (5,3) circle [radius=0.1];
\draw [fill] (8,3) circle [radius=0.1];
\draw [fill] (10,3) circle [radius=0.1];
\draw [fill] (12,3) circle [radius=0.1];
\node at (5,3.5) {$2$};
\node at (8,3.5) {$5$};
\node at (10,3.5) {$6$};
\node at (12,3.5) {$7$};
\draw (8,3)--(8,3)--(10,3)--(10,3)--(12,3);
\end{tikzpicture}
    \caption{$P_7^1$}\label{fig:p71}
\end{subfigure}
\vspace{0.3cm}

\begin{subfigure}[]{0.35\textwidth}
    \centering
    \begin{tikzpicture}[scale=.55]
\draw [fill] (4,3) circle [radius=0.1];
\draw [fill] (7,3) circle [radius=0.1];
\draw [fill] (10,3) circle [radius=0.1];
\draw [fill] (12,3) circle [radius=0.1];
\node at (4,3.5) {$1$};
\node at (7,3.5) {$4$};
\node at (10,3.5) {$6$};
\node at (12,3.5) {$7$};
\draw (10,3)--(12,3);
\end{tikzpicture}
    \caption{$P_7^2$}\label{fig:p72}
\end{subfigure}
\hfill
\begin{subfigure}[]{0.35\textwidth}
    \centering
    \begin{tikzpicture}[scale=.55]
\draw [fill] (5,3) circle [radius=0.1];
\draw [fill] (8,3) circle [radius=0.1];
\draw [fill] (10,3) circle [radius=0.1];
\draw [fill] (12,3) circle [radius=0.1];
\node at (5,3.5) {$1$};
\node at (8,3.5) {$5$};
\node at (10,3.5) {$6$};
\node at (12,3.5) {$7$};
\draw (8,3)--(10,3)--(10,3)--(12,3);
\end{tikzpicture}
 \caption{$P_7^3$}\label{fig:p73}
\end{subfigure}
\caption{Path graph and its subgraphs.}
\label{fig:link_example}
\end{figure}
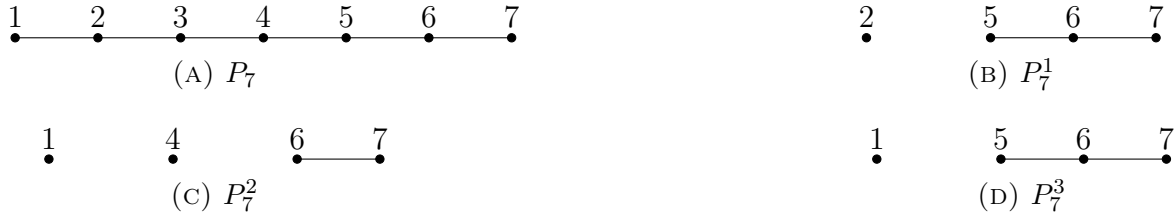

It is easy to see that the vertices labelled $3$ and $5$ are shedding vertices. 
The link of either of these vertices is a union of three subcomplexes. 
Each of this subcomplex is in fact an $\ind_2$ of an appropriate subgraph. 
For example, $\lk_{\Indr{2}(P_7)}(3)= \ind_2(P_7^1)\cup \ind_2(P_7^2)\cup \ind_2(P_7^3)$ (see \Cref{fig:link_example}). 
As $r$ and the number of vertices of $G$ increase the expression for the link of a vertex in $\ind_r(G)$ is more involved. 
Therefore, even if we assume, by induction, that each of these subcomplexes in the union are vertex decomposable, nothing can be concluded for the link itself. Hence we use ideas from commutative algebra for our proofs. 
We begin by recalling some related concepts.

Let $\K$ be a simplicial complex and  $V(\K)=\{x_1,x_2,\ldots,x_n\}$. Let $R=\mathbb{K}[x_1,\ldots,x_n]$ be the polynomial ring in $n$ indeterminates over a field $\mathbb{K}$. To $\K$ we associate  a square-free monomial ideal $\mathcal{SR}(\K)=\langle \mathbf x_F:=\prod_{x_i\in F} x_i\mid F \text{ is not a face of }\K \rangle$ in the polynomial ring $R$. The ideal $\mathcal{SR}(\K)$ is called the {\it Stanley-Reisner ideal} of $\K$. The ideal $\mathcal{SR}(\K)$ can also be viewed as a hyper-edge ideal of a hypergraph as follows.

Let $\mathcal{H}$ be a simple hypergraph on the vertex set $V=\{x_1,\ldots,x_n\}$.  The {\it hyper-edge ideal} of $\mathcal{H}$, denoted by $I(\mathcal{H})$, is a square-free monomial ideal in the polynomial ring $R=\mathbb{K}[x_1,\ldots,x_n]$ generated by the monomials $\mathbf{x}_f:=\prod_{x_i\in f}x_i$, where $f$ is an edge of $\mathcal{H}$. If $\K$ is a simplicial complex on the vertex set $V$, then the collection 
$$\mathcal{H}_{\K}=\{F\subseteq V\mid \mathbf{x}_F\text{ is a minimal generator of } \mathcal{SR}(\K)\}$$
forms a simple hypergraph on the
vertex set $V$ and $\mathcal{SR}(\K)$ is same as $I(\H_{\K})$.

Let $I=\langle x_{11}\cdots x_{1l_1},\ldots,x_{q1}\cdots x_{ql_q}\rangle$ be a squarefree monomial ideal in the polynomial ring $R$. The {\it Alexander dual ideal} $I^{\vee}$ of $I$ is defined as
\[
 I^{\vee}:= \langle x_{11},\ldots, x_{1l_1}\rangle\cap\dots\cap\langle x_{q1},\ldots ,x_{ql_q}\rangle.
\]

The {\it Alexander dual of $\K$} is the simplicial complex $\K^{\vee}=\{V\setminus F\mid F\notin\K\}$. From the definition we see that $(\mathcal{SR}(\K))^{\vee}=\langle \mathbf{x}_{V \setminus F} \mid F\text{ is a facet of }\K\rangle$
and hence $(\mathcal{SR}(\K))^{\vee}=\mathcal{SR}(\K^{\vee})$. Recently, Moradi and Khosh-Ahang \cite{MKA} defined the notion of vertex splittable ideal, which is an algebraic analog of the vertex decomposability property of a simplicial complex.

\begin{definition}
\normalfont
 Let $I$ be a monomial ideal in the polynomial ring $R=\mathbb K[x_1,\ldots,x_n]$. We say that $I$ is {\it vertex splittable} if $I$ can be obtained by the following recursive procedure.
 \begin{enumerate}
  \item[(i)] If $I=\langle m\rangle$ where $m$ is a monomial or $I=(0)$ or $I=R$, then $I$ is a vertex splittable ideal.
  
  \item[(ii)] If there exists a variable $x_i$ and two vertex splittable ideals $I_1$ and $I_2$ of $\mathbb K[x_1,\ldots,\widehat{x_i},\ldots,x_n]$ such that $I=x_iI_1+I_2$ with $I_2\subseteq I_1$ and the minimal generators of $I$ is the disjoint union of the minimal generators of $x_iI_1$ and $I_2$, then $I$ is a vertex splittable ideal. 
 \end{enumerate}
\end{definition}

Our aim is to show that the higher independence complexes of certain chordal graphs are vertex decomposable. To do so, we use the following theorem.

\begin{theorem}\textup{\cite[Theorem 2.3]{MKA}}
 A simplicial complex $\K$ is vertex decomposable if and only if $(\mathcal{SR}(\K))^{\vee}$ is vertex splittable.\label{Vertex splittable}
\end{theorem}

By previous discussion, $\mathcal{SR}(\mathcal K)$ can be interpreted as the hyper-edge ideal $I(\mathcal H_{\mathcal K})$. Furthermore, the minimal generators of the ideal $I(\mathcal H_{\mathcal K})^{\vee}$ can be described in terms of the minimal vertex covers of $\H_{\mathcal K}$. In general for a hypergraph $\H$, a subset $\mathcal{C} \subseteq V(\H)$ is called a {\it vertex cover} of $\mathcal H$ if $\mathcal{C}\cap f\neq\emptyset$ for each edge $f\in E(\mathcal{H})$. A vertex cover $\mathcal{C}$ of $\mathcal{H}$ is called a {\it minimal vertex cover} if it is minimal with respect to inclusion among all the vertex covers of $\mathcal{H}$.
 The following result  describes the ideal $I(\H)^\vee$ in terms of the minimal vertex covers of $\H$.

\begin{proposition}{\textup{\cite[cf. Theorem 6.3.39]{RV1}}}\label{min gen dual ideal}
 \[I(\H)^{\vee}=\langle x_{i_1}\cdots x_{i_t}\mid \{x_{i_1},\ldots,x_{i_t}\} \text{ is a minimal vertex cover of } \H \rangle.\]
 
\end{proposition}

We now proceed to show that the $r$-independence complex of a caterpillar graph is vertex decomposable. But first we recall some definitions. 

A {\itshape leaf}  is a vertex which  is adjacent to exactly one vertex. If a vertex is adjacent to at least two vertices, then it is called an {\itshape internal vertex}. An internal vertex is called a {\itshape corner vertex}, if it is adjacent to at most one internal vertex. A {\itshape tree} is a graph in which any two vertices are connected by exactly one path. The following graphs are a special class of trees.

\begin{definition}\label{def:caterpillar graph}
\normalfont A {\itshape caterpillar graph} is a tree in which every vertex is on a central path or only one edge away from the path (see \Cref{figure 122465} for an example). 
\end{definition}

In this section we show that the Alexander dual of Stanley-Reisner ideal of the $r$-independence complex of a caterpillar graph is a vertex splittable ideal.
Let $CG$ be a caterpillar graph with $|V(CG)|=n$ and $\{\alpha_i:i\in [l]\}$ be the set of all internal vertices of $CG$ such that $\alpha_i$ is connected to $\alpha_{i-1}$ and $\alpha_{i+1}$ for $2\le i\le l-1$. Moreover, $\alpha_1$ and $\alpha_l$ are connected to $\alpha_2$ and $\alpha_{l-1}$, respectively. Let $1\le t\le l
$ be an integer such that the leaves attached to $\alpha_t$ are $\beta_{t,1},\ldots,\beta_{t,m_t},\ldots,\beta_{t,m_t'}$ for some integers $ m_t\le m_{t'}$. Also for $1\le i\le l$ and $i\neq t$, the leaves attached to $\alpha_i$ are $\beta_{i,1},\ldots,\beta_{i,m_i}$ for some integers $m_i$. Thus for $2\le i\le l-1$ and $i\neq t$ we have $N(\alpha_i)=\{\alpha_{i-1},\alpha_{i+1},\beta_{i,1},\ldots,\beta_{i,m_i}\}$. Moreover, $N(\alpha_1)=\{\alpha_2,\beta_{1,1},\ldots,\beta_{1,m_1}\}$, $N(\alpha_l)=\{\alpha_{l-1},\beta_{l,1},\ldots,\beta_{l,m_l}\}$ and $N(\alpha_t)=\{\alpha_{t-1},\alpha_{t+1},\beta_{t,1},\ldots,\beta_{t,m_t},\ldots,\beta_{t,m_t'}\}$. 

\begin{remark}
\normalfont    Note that it may happen that some of the vertices $\alpha_i$ in the central path may not be connected to any leaf. In that case, the value of $m_i$ would be zero. Since the $m_i = 0$ cases do not affect any results or the calculations, for brevity, we are not explicitly writing those cases in the following calculations. When $m_i=0$, the set $B_{i,0}$ is the same as the set $A_i$ in the discussion below.
\end{remark}

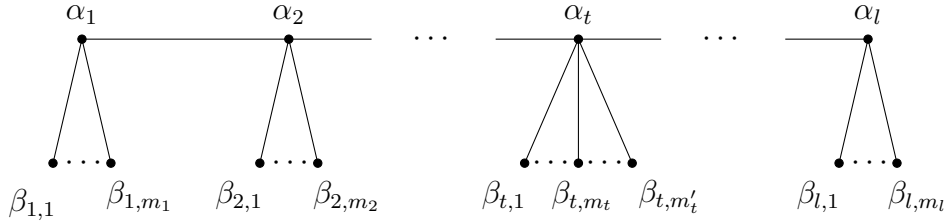
\begin{figure}[!ht]
\centering
\begin{tikzpicture}
[scale=.55]
\draw [fill] (0,3) circle [radius=0.1];
\draw [fill] (5,3) circle [radius=0.1];
\draw [fill] (12,3) circle [radius=0.1];
\draw [fill] (19,3) circle [radius=0.1];
\draw [fill] (-0.7,0) circle [radius=0.1];
\draw [fill] (0.7,0) circle [radius=0.1];
\draw [fill] (4.3,0) circle [radius=0.1];
\draw [fill] (5.7,0) circle [radius=0.1];
\draw [fill] (10.7,0) circle [radius=0.1];
\draw [fill] (12,0) circle [radius=0.1];
\draw [fill] (13.3,0) circle [radius=0.1];
\draw [fill] (18.3,0) circle [radius=0.1];
\draw [fill] (19.7,0) circle [radius=0.1];
\node at (0,3.6) {$\alpha_1$};
\node at (5,3.6) {$\alpha_2$};
\node at (12,3.6) {$\alpha_t$};
\node at (19,3.6) {$\alpha_l$};
\node at (-1.2,-1) {$\beta_{1,1}$};
\node at (1.4,-0.8) {$\beta_{1,m_1}$};
\node at (3.8,-0.8) {$\beta_{2,1}$};
\node at (6.4,-0.8) {$\beta_{2,m_2}$};
\node at (10.2,-0.8) {$\beta_{t,1}$};
\node at (12.1,-0.8) {$\beta_{t,m_t}$};
\node at (14.2,-0.8) {$\beta_{t,m_t'}$};
\node at (17.8,-0.8) {$\beta_{l,1}$};
\node at (20.2,-0.8) {$\beta_{l,m_l}$};
\node at (8.5,3) {$\cdots$};
\node at (15.5,3) {$\cdots$};
 \node at (0.1,0) {$\cdots$};
 \node at (5.1,0) {$\cdots$};
 \node at (11.4,0) {$\cdots$};
 \node at (12.7,0) {$\cdots$};
 \node at (19.1,0) {$\cdots$};
 \draw (0,3)--(-0.7,0);
 \draw (0,3)--(0.7,0);
\draw (5,3)--(4.3,0);
\draw (5,3)--(5.7,0);
\draw (12,3)--(10.7,0);
\draw (12,3)--(12,0);
\draw (12,3)--(13.3,0);
\draw (19,3)--(18.3,0);
\draw (19,3)--(19.7,0);
\draw (0,3)--(5,3);
\draw (5,3)--(7,3);
\draw (10,3)--(12,3);
 \draw (12,3)--(14,3);
 \draw (17,3)--(19,3);
\end{tikzpicture}\caption{Caterpillar graph}\label{figure 122465}
\end{figure}

Suppose $|\{\alpha_i,\beta_{i,j}\mid 1\le i\le t,\,1\le j\le m_i\}|=r+1$. Thus $r+1=\sum_{i=1}^tm_i+t$. Let $I=(\mathcal{SR}(\operatorname{Ind}_r(CG)))^{\vee}$ be the Alexander dual ideal in the polynomial ring $R=\mathbb{K}[x_{\theta}\mid \theta\in V(CG)]$. Observe that $I$ is the dual of the hyper-edge ideal of $\con_r(CG)$ (see \Cref{def:conrG}). By \Cref{min gen dual ideal}, the minimal generators of $I$ are given by the minimal vertex covers of $\con_r(CG)$, i.e., 
\[
I=\langle x_{\theta_1}\cdots x_{\theta_s}\mid \{\theta_1,\ldots,\theta_s\}\subseteq V(CG)\text{ is a minimal vertex cover of } \con_r(CG)\rangle .
\]

Let $\mathcal C$ be a minimal vertex cover of $\con_r(CG)$. Note that $\mathcal C$ must contain at least one element from the set $\{\alpha_i,\beta_{i,j}:1\le i\le t\text{ and } 1\le j\le m_i\}$. 

For $1\le i\le t$ and $0\le k\le m_i$ let \[
A_i:=\{\mathcal C\mid \alpha_i\in\mathcal C\text{ but } \alpha_j,\beta_{j,1},\ldots ,\beta_{j,m_j} \notin\mathcal{C}\text{ for }1\le j\le i-1\},\] and 
\begin{align*}
B_{i,k}:=\{\mathcal C\mid \beta_{i,k}\in\mathcal C \text{ but }\alpha_u,\alpha_i,\beta_{u,v},\beta_{i,w}\notin\mathcal C  \text{ for } 1\le u\le i-1,1\le v\le m_u,1\le w\le k-1\}.
\end{align*}
Here we make the convention that $\beta_{i,0}=\alpha_i$.

In the definitions of $A_i$ and $B_{i,k}$, $\mathcal{C}$ is always considered to be a minimal vertex cover of $\con_r(CG)$. Note that if $\mathcal C$ is a minimal vertex cover of $\con_r(CG)$ such that $\beta_{i,k}\in\mathcal C$ then $\alpha_u\notin\mathcal C$ for $i\le u\le t$. Moreover, if $\mathcal C$ is a minimal vertex cover of $\con_r(CG)$ such that $\alpha_i\in\mathcal{C}$ for some $i\le t$, then $\alpha_u\notin\mathcal C$ for $i<u\le t$.

For a subset $F=\{\theta_1,\ldots,\theta_s\}\subseteq V(CG)$, let $\mathbf x_{F}$ denote the monomial $x_{\theta_1}\cdots x_{\theta_s}$ in the polynomial ring $R$. Consider the ideals $J_i=\langle \mathbf x_{\mathcal C\setminus \{\alpha_i\}}\mid \mathcal C\in A_i\rangle$ for $1\le i\le t$ and $J_{i,k}=\left\langle \mathbf x_{\mathcal C\setminus \{\beta_{i,k}\}}\mid \mathcal C\in B_{i,k}\right\rangle$ for $1\le k\le m_i$ in $R$. Then

\begin{align}\label{ideal decomposition}
I=\sum_{i=1}^t x_{\alpha_i}\cdot J_i+\sum_{i=1}^t\sum_{k=1}^{m_i}x_{\beta_{i,k}}\cdot J_{i,k}.
\end{align}
Here we remark that if a minimal vertex cover of $\con_r(CG)$ contains $\alpha_i$ for some $1\le i\le t$, then it automatically does not contain $\beta_{i,k}$ for each $1\le k\le m_i$. 

\begin{lemma}\label{leaf and internal}
 With the notations as above we have $J_{1,1}\subseteq J_1$
\end{lemma}
\begin{proof}
 Let $\mathcal{C}$ be a minimal vertex cover of $\con_r(CG)$ such that $\mathcal{C}\in B_{1,1}$. It is enough to show that $\mathcal{C}'=(\mathcal{C}\setminus\{\beta_{1,1}\})\cup\{\alpha_1\}$ is a vertex cover of $\con_r(CG)$. Let $S\subseteq V(CG)$ be such that $|S|=r+1$ and the induced subgraph $CG[S]$ is connected. We need to show that $S\cap\mathcal C'\neq \emptyset$. This clearly holds because if $S\subseteq V(CG)$ such that $|S|=r+1$ and $CG[S]$ is connected then $\beta_{1,1}\in S$ implies $\alpha_1\in S$. Now since $S\cap\mathcal{C}\neq\emptyset$ we have $S\cap\mathcal{C}'\neq\emptyset$.
\end{proof}

\begin{lemma}\label{leaf ideal containment}
 With the notations as above we have $J_{i,k}\subseteq J_{i',k'}$ for $i\ge i'$ and $k\ge k'$.
\end{lemma}

\begin{proof}
 Let $\mathcal{C}$ be a minimal vertex cover of $\con_r(CG)$ such that $\mathcal{C}\in B_{i,k}$. We need to show that $\mathcal{C}'=(\mathcal{C}\setminus\{\beta_{i,k}\})\cup\{\beta_{i',k'}\}$ is a vertex cover of $\con_r(CG)$. Let $S\subseteq V(CG)$ such that $|S|=r+1$ and the induced subgraph $CG[S]$ is connected. If $\beta_{i',k'}\in S$ then $\mathcal{C}'\cap S\neq\emptyset$. Thus we can assume that $\beta_{i',k'}\notin S$. If $\beta_{i,k}\notin S$, then as $\mathcal{C}\cap S\neq\emptyset$ we also have $\mathcal{C}'\cap S\neq\emptyset$. Thus we can also assume that $\beta_{i,k}\in S$. Hence $\alpha_i\in S$ since $|S|=r+1\ge 2$ and $CG[S]$ is connected. If $\alpha_s\in S$ for each $s\le i$, then we see that $CG[S']$ is a connected graph for $S'=(S\setminus \{\beta_{i,k}\})\cup\{\beta_{i',k'}\}$, where $|S'|=r+1$. In that case $\mathcal{C}\cap S'\neq\emptyset$ and consequently $\mathcal{C}'\cap S\neq\emptyset$. Let $1<q\le i$ be the least positive integer such that $\alpha_q\in S$ and $\alpha_{q-1}\notin S$. In this case consider $S'=(S\setminus\{\beta_{i,k}\})\cup\{\alpha_{q-1}\}$. Since $CG[S]$ is a connected graph we see that $CG[S']$ is also a connected graph such that $|S'|=r+1$. Therefore, $\mathcal{C}\cap S'\neq\emptyset$ and hence $\mathcal{C}'\cap S\neq\emptyset$. Thus $\mathcal{C}'$ is a vertex cover of $\con_r(CG)$.
\end{proof}

The proof of the following lemma is similar to that of \Cref{leaf ideal containment} and hence we omit.

\begin{lemma}\label{main ideal decomposition}
 With the notations as above we have $J_{i'}\subseteq J_i$ for $i\ge i'$.
\end{lemma}

Continuing with the notations as above, for $1\le i\le t$, let $CG^{(i+1)}$ be the graph obtained from $CG$ by removing the vertices $\{\alpha_u,\beta_{u,v}\mid 1\le u\le i,\,1\le v\}$. Let $I^{(i+1)}=(\mathcal{SR}(\operatorname{Ind}_r(CG^{(i+1)})))^{\vee}$ be the corresponding Alexander dual ideal in $R^{(i+1)}=\mathbb{K}[x_{\theta}\mid  \theta\in V(CG^{(i+1)})]$. Observe that $I^{(i+1)}$ is the dual of the hyper-edge ideal of $\con_r(CG^{(i+1)})$. Then as in \Cref{ideal decomposition}, we can write the ideal 

\begin{align}\label{J_i equation expression}
 I^{(i+1)}=\sum_{j\ge i+1}x_{\alpha_j}\cdot J_{j}^{(i+1)}+\sum_{j\ge i+1}\sum_{k\ge 1}x_{\beta_{j,k}}\cdot J_{j,k}^{(i+1)}
\end{align}
in the polynomial ring $R^{(i+1)}$. Note that for each $\alpha_j$ in \Cref{J_i equation expression}, 
\begin{align}\label{cardinality1}
|\{\alpha_u,\alpha_j,\beta_{u,v}\mid i+1\le u<j,1\le v\}|\le r+1
\end{align}
and also for each $\beta_{j,k}$ in \Cref{J_i equation expression},
\begin{align}\label{cardinalry2}
|\{\alpha_u,\alpha_j,\beta_{u,v},\beta_{j,w}\mid i+1\le u<j,1\le v,1\le w\le k\}|\le r+1.
\end{align}
Moreover, as in  \Cref{ideal decomposition}, $J_j^{(i+1)}$ is generated by the monomials $\mathbf{x}_{\mathcal{C}'}$ such that $\mathcal{C}'\cup\{\alpha_j\}$ is a minimal vertex cover of $\con_r(CG^{(i+1)})$ which does not contain any element of the set $\{\alpha_u,\beta_{u,v}\mid u<j,v\ge 1\}$. Similarly, if $\mathbf{x}_{\mathcal{C}'}$ is a minimal generator of $J_{j,k}^{(i+1)}$ then $\mathcal{C}'\cup\{\beta_{j,k}\}$ is a minimal vertex cover of $\con_r(CG^{(i+1)})$ which does not contain any element of the set $\{\alpha_u,\alpha_j,\beta_{u,v},\beta_{j,w}\mid 1\le u<j,v\ge 1,1\le w\le k-1\}$.

\begin{lemma}\label{Ji decomposition}
 With the notations as above, for $1\le i\le t$ we have
 \begin{align}\label{decomposition of Ji}
  J_i=\sum_{j>t}x_{\alpha_j}\cdot J_{j}^{(i+1)}R+\sum_{j\ge i+1}\sum_{k\ge 1}x_{\beta_{j,k}}\cdot J_{j,k}^{(i+1)}R,
 \end{align}
 where $\alpha_j$ and $\beta_{j,k}$ also satisfy \Cref{cardinality1} and \Cref{cardinalry2}, respectively.
\end{lemma}
\begin{proof}
 Let us simply denote the ideal in the right hand side of \Cref{decomposition of Ji} as $I'$. Notice that $I'$ is generated by the monomials $\mathbf{x}_{\mathcal{C}'}$, where $\mathcal{C}'$ is a minimal vertex cover of $\con_r(CG^{(i+1)})$ which does not contain $\alpha_j$ for $j\le t$. Let $\mathcal{C}$ be a minimal vertex cover of $\con_r(CG)$ such that $\mathcal{C}\in A_i$. By construction $\alpha_j\notin \mathcal{C}$ for $j<i$. Also, if $\alpha_j\in\mathcal{C}$ for some $i<j\le t$, then $\mathcal{C}\setminus\{\alpha_i\}$ is a vertex cover of $\con_r(CG)$, contradiction to the fact that $\mathcal{C}$ is a minimal vertex cover of $\con_r(CG)$. Thus we see that $\mathcal{C}\setminus\{\alpha_i\}$ is a vertex cover of $\con_r(CG^{(i+1)})$. Since $J_i$ is generated by the monomials $\mathbf{x}_{\mathcal C\setminus \{\alpha_i\}}$, we deduce that $J_i\subseteq I'$.
 
 Now let $\mathcal{C}$ be a minimal vertex cover of $\con_r(CG^{(i+1)})$ such that $\alpha_j\in\mathcal{C}$ for some $j>t$, where $\alpha_j$ satisfies \Cref{cardinality1}, and $\alpha_u,\beta_{u,v}\notin\mathcal{C}$ for $i+1\le u<j$ and $v\ge 1$. As $r+1=\sum_{i=1}^t m_i+t$, we see that $\mathcal{C}$ is not a vertex cover of $\con_r(CG)$. Therefore, $\mathcal{C}\cup\{\alpha_i\}$ is minimal vertex cover of $\con_r(CG)$. Hence, $x_{\alpha_j}\cdot J_{j}^{(i+1)}R\subseteq J_i$ for $j>t$. Next suppose that $\mathcal{C}$ is a minimal vertex cover of $\con_r(CG^{(i+1)})$ such that $\beta_{j,k}\in\mathcal{C}$ for some $j\ge i+1,k\ge 1$; where $\beta_{j,k}$ satisfies \Cref{cardinalry2} and $\mathcal{C}$ does not contain any element of the set $\{\alpha_u,\beta_{u,v},\beta_{j,w}\mid i+1\le u\le j,v\ge 1,1\le w\le k-1\}$. We show that $\mathcal{C}$ is not a vertex cover of $\con_r(CG)$. Let $s=\max\{u\mid\alpha_u\in\con_r(CG^{(i+1)})\text{ and }\alpha_u\notin\mathcal{C}\}$. Take $S=W\setminus\mathcal{C}$, where $W=\{\alpha_u,\beta_{u,v}\mid i+1\le u\le s,v\ge 1\}$. If $|S|\ge r+1$, then $CG^{(i+1)}[S]$ is a connected subgraph of $CG^{(i+1)}$ such that $S\cap\mathcal{C}=\emptyset$, a contradiction. If $|S|\le r-1$, then $\mathcal{C}\setminus\{\beta_{j,k}\}$ is a vertex cover of $CG^{(i+1)}$, a contradiction. Thus $|S|=r$. Taking $S'=S\cup\{\alpha_i\}$ we see that $CG[S']$ is a connected subgraph of $CG$ with $|S'|=r+1$ and $S'\cap\mathcal{C}=\emptyset$. Thus $\mathcal{C}$ is not a vertex cover of $\con_r(CG)$. Consequently, $\mathcal{C}\cup\{\alpha_i\}$ is a minimal vertex cover of $\con_r(CG)$. Hence, $x_{\beta_{j,k}}\cdot J_{j,k}^{(i+1)}R\subseteq J_i$ and this completes the proof.
\end{proof}

Let $\widetilde{CG}$ be the graph obtained from $CG$ by removing the vertex $\beta_{1,1}$. Let $\widetilde I=(\mathcal{SR}(\operatorname{Ind}_r(\widetilde{CG})))^{\vee}$ be the corresponding Alexander dual ideal in $\widetilde R=\mathbb{K}[x_{\theta}\mid \theta\in V(\widetilde{CG})]$. Observe that $\widetilde I$ is the dual of the hyper-edge ideal of $\con_r(\widetilde{CG})$. Then as in  \Cref{ideal decomposition}, we can write the ideal 

\begin{align*}
\small{
 \widetilde I=\begin{cases}
 \sum_{j= 1}^tx_{\alpha_j}\cdot \widetilde{J}_j+\sum_{p= 2}^{m_1}x_{\beta_{1,p}}\cdot \widetilde{J}_{1,p}+\sum_{j=2 }^t\sum_{p= 1}^{m_j}x_{\beta_{j,p}}\cdot \widetilde{J}_{j,p}+x_{\beta_{t,m_t+1}}\cdot\widetilde J_{t,m_t+1}&\text{ if }m_{t'}>m_t\\
 \sum_{j= 1}^{t+1}x_{\alpha_j}\cdot \widetilde{J}_j+\sum_{p= 2}^{m_1}x_{\beta_{1,p}}\cdot \widetilde{J}_{1,p}+\sum_{j=2 }^t\sum_{p= 1}^{m_j}x_{\beta_{j,p}}\cdot \widetilde{J}_{j,p}&\text{ if }m_{t'}=m_t
 \end{cases}
 }
\end{align*}
in the polynomial ring $\widetilde R$. Here as in \Cref{ideal decomposition}, $\widetilde{J}_j$ is generated by the monomials $\mathbf{x}_{\mathcal{C}'}$ such that $\mathcal{C}'\cup\{\alpha_j\}$ is a minimal vertex cover of $\con_r(\widetilde{CG})$ which does not contain any element of the set $\{\alpha_u,\beta_{u,v}\mid u<j,v\ge 1\}$. Similarly, if $\mathbf{x}_{\mathcal{C}'}$ is a minimal generator of $\widetilde{J}_{j,p}$, then $\mathcal{C}'\cup\{\beta_{j,p}\}$ is a minimal vertex cover of $\con_r(\widetilde{CG})$ which does not contain any element of the set $\{\alpha_u,\beta_{u,v},\beta_{j,w}\mid 1\le u<j,v\ge 1,1\le w\le p-1\}$.

\begin{lemma}\label{Jik decomposition}
 With the notations as above, for $1\le i\le t$ and $1\le k \le m_i$ we have 
 \begin{align}\label{other expression}
  J_{i,k}=\begin{cases}
    \sum_{p=k+1}^{m_i}x_{\beta_{i,p}}\cdot \widetilde{J}_{i,p}R+\sum_{j= i+1}^{t-1}\sum_{p= 1}^{m_j}x_{\beta_{j,p}}\cdot \widetilde{J}_{j,p}R+\sum_{p=1}^{m_t+1}x_{\beta_{t,p}}\cdot \widetilde{J}_{t,p}R&\text{ if }\, m_{t'}>m_t \\           
           \sum_{p=k+1}^{m_i}x_{\beta_{i,p}}\cdot \widetilde{J}_{i,p}R+\sum_{j= i+1}^{t}\sum_{p=1}^{m_j}x_{\beta_{j,p}}\cdot \widetilde{J}_{j,p}R+x_{\alpha_{t+1}}\cdot \widetilde{J}_{t+1}R&\text{ if }m_{t'}=m_t
     \end{cases}
 \end{align}
In particular, note that the above expression also implies $J_{i,k}\subseteq J_{i',k'}$ for $i\ge i'$ and $k\ge k'$. 
\end{lemma}
\begin{proof}
 Let $I''$ denote the ideal in the right hand side of  \Cref{other expression}. Suppose $\mathcal{C}$ is a minimal vertex cover of $\con_r(CG)$ such that $\mathcal{C}\in B_{i,k}$. By construction $\alpha_j\notin\mathcal{C}$ for $j\le t$. Let $s=\max\{w\mid\alpha_w\in\con_r(CG)\text{ and }\alpha_w\notin\mathcal{C}\}$. Take $S=W\setminus\mathcal{C}$, where $W=\{\alpha_u,\beta_{u,v}\mid 1\le u\le s,v\ge 1\}$. If $|S|\le r-1$, then $\mathcal{C}\setminus \{\beta_{i,k}\}$ becomes a vertex cover of $\con_r(CG)$, a contradiction. Moreover, if $|S|\ge r+1$, then $CG[S]$ is a connected subgraph of $CG$ and $\mathcal C\cap S=\emptyset$, again a contradiction. Thus $|S|=r$ and hence $\mathcal{C}\setminus\{\beta_{i,k}\}$ is a minimal vertex cover of $\con_r(\widetilde{CG})$. Therefore, $J_{i,k}\subseteq I''$. Now let $\mathbf{x}_{\mathcal{C}'}$ be a minimal generator of $I''$. Then $\mathcal{C}'$ is a minimal vertex cover of $\con_r(\widetilde{CG})$ such that $\alpha_j\notin \mathcal{C}'$ for $j\le t$. Let $s'=\max\{w\mid\alpha_w\in\con_r(\widetilde{CG})\text{ and }\alpha_w\notin\mathcal{C}'\}$ and $S'=W'\setminus\mathcal{C}'$, where $W'=\{\alpha_u,\beta_{u,v}\mid 1\le u\le s,v\ge 1\}$. Then proceeding as above we see that $|S'|=r$ and hence $\mathcal{C}'\cup\{\beta_{i,k}\}$ is a minimal vertex cover of $\con_r(CG)$. Consequently, $I''\subseteq J_{i,k}$.
\end{proof}

\begin{theorem}\label{main theorem *}
 Let $CG$ be the caterpillar graph on $n$ vertices. Then  $\operatorname{Ind}_r(CG)$ is a vertex decomposable simplicial complex for $r\ge 1$.
\end{theorem}
\begin{proof}
 By \Cref{Vertex splittable}, we need to show that $I=(\mathcal{SR}(\operatorname{Ind}_r(CG)))^{\vee}$ is a vertex splittable ideal. Continuing with the notation as above we have the expression of $I$ in  \Cref{ideal decomposition}. Observe that by \Cref{leaf and internal}, \Cref{leaf ideal containment} and \Cref{main ideal decomposition} we just need to show that the ideals $J_i$ and $J_{i,k}$ are vertex splittable. We show this by induction on $n$. Consider the caterpillar graph $CG^{(i+1)}$ and the Alexander dual ideal $I^{(i+1)}$ of the Stanley-Reisner ideal of $r$-independence complex of $CG^{(i+1)}$. Proceeding as in  \Cref{main ideal decomposition},  \Cref{leaf ideal containment} and  \Cref{leaf and internal} we see that $J_{j'}^{(i+1)}\subseteq J_j^{(i+1)}$ for $j\ge j'$, $J_{j,k}^{(i+1)}\subseteq J_{j',k'}^{(i+1)}$ for $j\ge j'$ and $k\ge k'$, and $J_{i+1,1}^{(i+1)}\subseteq J_{i+1}^{(i+1)}$, respectively. Thus by induction hypothesis and by  \Cref{Ji decomposition}, $J_i$ is a vertex splittable ideal for each $i$. Now consider the caterpillar graph $\widetilde{CG}$ and the Alexander dual ideal $\widetilde I$ of the Stanley-Reisner ideal of $r$-independence complex of $\widetilde{CG}$. Again, one can check that $\widetilde{J_{j,p}}\subseteq \widetilde{J_{j',p'}}$ for $j\ge j'$ and $p\ge p'$. Moreover, $\widetilde{J}_{j'}\subseteq \widetilde{J}_j$ for $j\ge j'$ and $\widetilde{J}_{1,2}\subseteq \widetilde{J}_1$. Hence by induction hypothesis and \Cref{Jik decomposition}, $J_{i,k}$ is a vertex splittable ideal. This completes the proof.
\end{proof}

We illustrate the proof of the above theorem with an example below.
\begin{example}
\normalfont
Let $CG$ be the caterpillar graph on the vertex set $\{\alpha_i,\beta_{1,1},\beta_{2,1},\beta_{2,2},\beta_{3,1},\beta_{4,1}\mid 1\le i\le 4\}$. Let $CG^{(2)}=CG\setminus\{\alpha_1,\beta_{1,1}\}$ and $CG^{(3)}=CG\setminus\{\alpha_1,\alpha_2,\beta_{1,1},\beta_{2,1},\beta_{2,2}\}$. Moreover, let $\widetilde{CG}=CG\setminus\{\beta_{1,1}\}$.
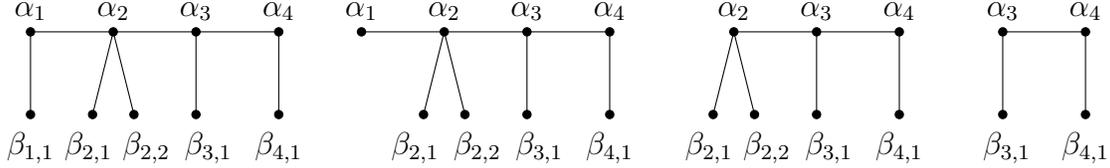
 \begin{figure}[ht]
\centering
\begin{tikzpicture}
[scale=.55]
\draw [fill] (0,1) circle [radius=0.1];
\draw [fill] (0,3) circle [radius=0.1];
\draw [fill] (1.5,1) circle [radius=0.1];
\draw [fill] (2.5,1) circle [radius=0.1];
\draw [fill] (2,3) circle [radius=0.1];
\draw [fill] (4,1) circle [radius=0.1];
\draw [fill] (4,3) circle [radius=0.1];
\draw [fill] (6,1) circle [radius=0.1];
\draw [fill] (6,3) circle [radius=0.1];
\node at (0,3.5) {$\alpha_1$};
\node at (2,3.5) {$\alpha_2$};
\node at (4,3.5) {$\alpha_3$};
\node at (6,3.5) {$\alpha_4$};
\node at (0,0.2) {$\beta_{1,1}$};
\node at (1.4,0.2) {$\beta_{2,1}$};
\node at (2.8,0.2) {$\beta_{2,2}$};
\node at (4.3,0.2) {$\beta_{3,1}$};
\node at (6,0.2) {$\beta_{4,1}$};
\draw (0,3)--(4,3)--(4,3)--(6,3);
\draw (0,1)--(0,3);
\draw (1.5,1)--(2,3);
\draw (2.5,1)--(2,3);
\draw (4,1)--(4,3);
\draw (6,1)--(6,3);

\draw [fill] (8,3) circle [radius=0.1];
\draw [fill] (9.5,1) circle [radius=0.1];
\draw [fill] (10.5,1) circle [radius=0.1];
\draw [fill] (10,3) circle [radius=0.1];
\draw [fill] (12,1) circle [radius=0.1];
\draw [fill] (12,3) circle [radius=0.1];
\draw [fill] (14,1) circle [radius=0.1];
\draw [fill] (14,3) circle [radius=0.1];
\node at (8,3.5) {$\alpha_1$};
\node at (10,3.5) {$\alpha_2$};
\node at (12,3.5) {$\alpha_3$};
\node at (14,3.5) {$\alpha_4$};
\node at (9.3,0.2) {$\beta_{2,1}$};
\node at (10.8,0.2) {$\beta_{2,2}$};
\node at (12.3,0.2) {$\beta_{3,1}$};
\node at (14,0.2) {$\beta_{4,1}$};
\draw (8,3)--(12,3)--(12,3)--(14,3);
\draw (9.5,1)--(10,3);
\draw (10.5,1)--(10,3);
\draw (12,1)--(12,3);
\draw (14,1)--(14,3);

\draw [fill] (16.5,1) circle [radius=0.1];
\draw [fill] (17.5,1) circle [radius=0.1];
\draw [fill] (17,3) circle [radius=0.1];
\draw [fill] (19,1) circle [radius=0.1];
\draw [fill] (19,3) circle [radius=0.1];
\draw [fill] (21,1) circle [radius=0.1];
\draw [fill] (21,3) circle [radius=0.1];
\node at (17,3.5) {$\alpha_2$};
\node at (19,3.5) {$\alpha_3$};
\node at (21,3.5) {$\alpha_4$};
\node at (16.4,0.2) {$\beta_{2,1}$};
\node at (17.8,0.2) {$\beta_{2,2}$};
\node at (19.3,0.2) {$\beta_{3,1}$};
\node at (21,0.2) {$\beta_{4,1}$};
\draw (17,3)--(19,3)--(21,3);
\draw (16.5,1)--(17,3);
\draw (17.5,1)--(17,3);
\draw (19,1)--(19,3);
\draw (21,1)--(21,3);

\draw [fill] (23.5,1) circle [radius=0.1];
\draw [fill] (23.5,3) circle [radius=0.1];
\draw [fill] (25.5,1) circle [radius=0.1];
\draw [fill] (25.5,3) circle [radius=0.1];
\node at (23.5,3.5) {$\alpha_3$};
\node at (25.5,3.5) {$\alpha_4$};
\node at (23.6,0.2) {$\beta_{3,1}$};
\node at (25.5,0.2) {$\beta_{4,1}$};
\draw (23.5,3)--(25.5,3);
\draw (23.5,1)--(23.5,3);
\draw (25.5,1)--(25.5,3);
\end{tikzpicture}\label{figure 12465}
\end{figure}

For $r=3$, let $I=(\mathcal{SR}(\operatorname{Ind}_r(CG)))^{\vee}$, $I^{(2)}=(\mathcal{SR}(\operatorname{Ind}_r(CG^{(2)})))^{\vee}$, $I^{(3)}=(\mathcal{SR}(\operatorname{Ind}_r(CG^{(3)})))^{\vee}$ and $\widetilde I=(\mathcal{SR}(\operatorname{Ind}_r(\widetilde{CG})))^{\vee}$. Using Macaulay 2 \cite{MAC2} one can check the following:

$I^{(2)}=\langle x_{\alpha_2}\cdot (x_{\alpha_4},x_{\beta_{3,1}},x_{\beta_{4,1}})\rangle+ I^{(2)}_1$, where $I_1^{(2)}=\langle x_{\beta_{2,1}}\cdot ( x_{\alpha_4}x_{\beta_{2,2}},x_{\alpha_4} x_{\beta_{3,1}},x_{\beta_{2,2}} x_{\beta_{3,1}} x_{\beta_{4,1}})$, $x_{\beta_{2,2}}\cdot(x_{\alpha_4}x_{\beta_{3,1}})$, $x_{\alpha_3}\rangle$; $I^{(3)}=\langle x_{\alpha_3},x_{\beta_{3,1}},x_{\alpha_4},x_{\beta_{4,1}}\rangle$ and
$\widetilde I=\langle x_{\alpha_1}\cdot (x_{\alpha_3},x_{\alpha_4}x_{\beta_{2,1}}x_{\beta_{2,2}},x_{\alpha_4}x_{\beta_{2,1}}x_{\beta_{3,1}}$, $x_{\alpha_4}x_{\beta_{2,2}}x_{\beta_{3,1}},x_{\beta_{2,1}}x_{\beta_{2,2}}x_{\beta_{3,1}}x_{\beta_{4,1}}),x_{\alpha_2}\cdot (x_{\alpha_3},x_{\alpha_4},x_{\beta_{3,1}},  x_{\beta_{4,1}})\rangle+\widetilde I_1+\widetilde I_2$, where $\widetilde I_1=\langle x_{\beta_{2,1}}\cdot (x_{\alpha_3},x_{\alpha_4}x_{\beta_{2,2}}x_{\beta_{3,1}} )\rangle$ and $\widetilde I_2=\langle x_{\beta_{2,2}}\cdot(x_{\alpha_3}) \rangle$. Here for monomials $m,m_1,\ldots,m_r$, the ideal $\langle mm_1,\ldots,mm_r\rangle$ is also denoted by $\langle m(m_1,\ldots,m_r)\rangle$. Now we can check that the ideal
\[
I=x_{\alpha_1}\cdot J_1+x_{\alpha_2}\cdot J_2+x_{\beta_{1,1}}\cdot J_{1,1}+x_{\beta_{2,1}}\cdot J_{2,1},\]
where $J_1=I_1^{(2)}$, $J_2=I^{(3)}$, $J_{1,1}=\widetilde I_1+\widetilde I_2$ and $J_{2,1}=\widetilde I_2$.
\end{example}

One particular example of a caterpillar graph is the path graph $P_n$. 
Thus, the following result is an immediate consequence of \Cref{main theorem *}.
\begin{corollary}
For $r\geq 1$, the $r$-independence complex of $P_n$ is vertex decomposable.
\end{corollary}

Note that caterpillar graphs are special (recursive) trees, in the sense that the induced subgraph obtained by removing a leaf is again a caterpillar graph. 
This recursive property played an important role in the proof of \Cref{main theorem *}.
Using the computer algebra system Macaulay2 \cite{MAC2} we verified that for all non-caterpillar trees with at most 10 vertices the associated $r$-independence complexes are vertex decomposable. 
Based on these calculations we propose the following conjecture. 

\begin{conj}\label{con1}
For any tree $T$ and $r\geq 1$, the complex $\mathrm{Ind}_r(T)$ is vertex decomposable. 
\end{conj}

In \Cref{sec:concluding}, we discuss why the conjecture can not extend to all chordal graphs.

\section{Concluding remarks}\label{sec:concluding}

It is known that the $1$-independence complexes of chordal graphs are vertex decomposable \cite{DE09, Woodroofe09}. In this section we show that this is not the case in general, {\itshape i.e.}, for chordal graph $G$ and $r\geq 2$, $\ind_r(G)$ need not be vertex decomposable.
We begin with the definition of sequentially Cohen-Macaulay complexes. 
\begin{definition}{\cite[Exercise 4.1.5]{wachs07}}\label{def:scm}
\normalfont
Let $\K$ be a simplicial complex.  For $m \in \{1, 2,\dots ,$ ${\rm dim}(\K)\}$, the \emph{pure} $m$-skeleton $\K^{[m]}$ of a simplicial complex $\K$ is the subcomplex generated by all faces of dimension $m$. 
A simplicial complex $\K$ is sequentially Cohen-Macaulay (SCM for short) if and only if $\K^{[m]}$ is Cohen-Macaulay for all $m = 1,\dots, \dim(\K)$. 
\end{definition}

Now we construct our first graph. 
Let $r\geq 2$ and $H_r$ be the graph with $V(H_r) = \{v_1,..,v_r,v_{r+1},...,v_{2r},x_1,x_2\}$ such that the subgraph induced by $V(H_r) \setminus \{x_1,x_2\}$ is isomorphic to the complete graph $K_{2r}$. 
Now add edges such that $x_1$ is adjacent only to $v_1,..,v_r$ and $x_2$ is adjacent only to $v_{r+1},...,v_{2r}$.
It is not hard to see that the graph $H_r$ is chordal. 
One way to see this is that any induced subgraph on 4 vertices not isomorphic to $P_4$ contains a triangle.

\begin{proposition}
With the notations as before the complex $\mathrm{Ind}_{r+1}(H_r)$ is not SCM.  
\end{proposition}
\begin{proof}
For simplicity, denote $\mathrm{Ind}_{r+1}(H_r)$ by $\K$. 
By definition every $(r+1)$-subset of $V(H_r)$ describes a simplex of $\K$. 
Hence, $\K^{[m]}$ is the $m$-skeleton of the $(2r+2)$-simplex and hence it is Cohen-Macaulay for $m = 1,\dots, r$.
Now, in dimension $r+1$ - there are exactly two subsets of cardinality $r+2$, namely $\{v_1,\dots,v_r,x_1,x_2\}$ and $\{v_{r+1},\dots,v_{2r},x_1,x_2\}$. 
Therefore the complex $\K^{[r+1]}$ is the union of two $(r+1)$-simplices along an edge. Since $r\geq 2$, $\K^{[r+1]}$ is not Cohen-Macaulay implying that the complex $\mathrm{Ind}_{r+1}(H_r)$ is not SCM. 
\end{proof}

It is easy to see that the complex $\ind_{r+1}(H_r)$ has the homotopy type of a wedge of spheres of dimension $r$. 
For our next example we will construct a contractible complex. 


\begin{proposition}
 For $r\geq 2$, let $G_r$ be the graph in \Cref{fig:Gr}. Then $\ind_r(G_r)$ is contractible but not SCM.
\end{proposition}
\begin{proof}
Observe that the subcomplex $\lk_{\ind_r(G_r)}(r+1)$ is a cone with apex $1$, hence contractible. Therefore, $\ind_r(G_r)\simeq \del_{\ind_r(G_r)}(r+1) = \ind_r(G_r\setminus \{r+1\})$. Clearly, $\ind_r(G_r\setminus \{r+1\})$ is a join of $\ind_r(G_r[\{1,\dots,r,a,b\}])$ with a simplex on vertex set $\{r+2,\dots,2r\}$ implying that $\ind_r(G_r)$ is contractible.

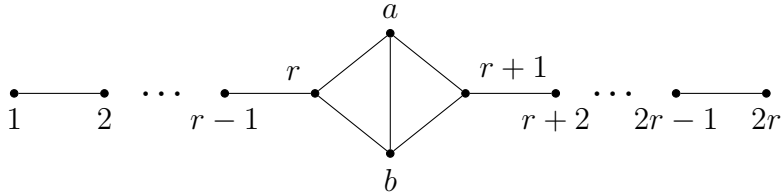
\begin{figure}[H]
		\centering
		\begin{tikzpicture}
 [scale=0.4, vertices/.style={draw, fill=black, circle, inner sep=1.0pt}]
        \node[vertices, label=below: {$1$}] (l1) at (0,0)  {};
		\node[vertices, label=below: {$2$}] (l2) at (3,0)  {};
		\node[vertices, label=below:
		{$r-1$}] (lr1) at (7,0)  {};
		\node[vertices, label=north west:
		{$r$}] (lr) at (10,0)  {};
		
		\node[vertices, label=above:
		{$a$}] (a) at (12.5,2)  {};\node[vertices, label=below:
		{$b$}] (b) at (12.5,-2)  {};
		
		\node[vertices, label=north east:
		{$r+1$}] (r1) at (15,0)  {};
		\node[vertices, label=below: {$r+2$}] (r2) at (18,0)  {};
		\node[vertices, label=below:
		{$2r-1$}] (rr1) at (22,0)  {};
		\node[vertices, label=below:
		{$2r$}] (rr) at (25,0)  {};
		\node at (5,0) {\bf \ldots};
		\node at (20,0) {\bf \ldots};
		
\foreach \to/\from in {l1/l2,lr1/lr,r1/r2,rr1/rr,lr/a,lr/b,a/b,a/r1,b/r1}
\draw [-] (\to)--(\from);
\end{tikzpicture}
\caption{The graph $G_r$}\label{fig:Gr}
	\end{figure}
	
We now prove that the complex $\ind_r(G_r)$ is not SCM by showing that the pure $(2r)$-skeleton of $\ind_r(G_r)$ is not Cohen-Macaulay. For simplicity of notations, denote $(\ind_r(G_r))^{[2r]}$ by $X_{2r}$. Let $F = \{1,\dots,r-1,r+2,\dots,2r\}$. Clearly $F \in \ind_r(G_r)$. 
Observe that $\lk_{X_{2r}}(F)$ is generated by two disjoint facets $\{r,r+1\}$ and $\{a,b\}$ implying that $\tilde{H}_0(\lk_{X_{2r}}(F);\mathbb{Z})\neq 0$. Since $\lk_{X_{2r}}(F)$ is of dimension $1$ with non-trivial homology in dimension $0$, $X_{2r}$ is not Cohen-Macaulay. Hence the complex  $\ind_r(G_r)$ is not sequentially Cohen-Macaulay.
\end{proof}
\begin{remark}
In the graph $G_r$ defined above, we can in fact take any complete graph $K_n$ ($n\geq 2$) in place of $G_r[\{a,b\}] \cong K_2 $ and add edges from vertices $r$ and $r+1$ to every vertex of $K_n$. Readers can verify, using similar arguments as in the previous result, that the $r$-independence complex of the new graph will again be contractible but not SCM. This way we get an infinite family of chordal graphs with contractible but not SCM $r$-independence complexes.
\end{remark}

We end the article with the following problems. 
\begin{enumerate}
    \item Find examples of chordal graphs $G$ and $r\geq 2$ such that $\ind_r(G)$ is SCM but not shellable.
    \item Find examples of chordal graphs $G$ and $r\geq 2$ such that $\ind_r(G)$ is shellable but not vertex decomposable. 
    \item Characterize chordal graphs such that $\ind_r(G)$ is vertex decomposable for all $r$. \item If $k\neq r+1$ then show that $\ind_k(H_r)$ is vertex decompsable. 
    \item Characterize $k$ such that $\ind_k(G_r)$ is vertex decomposable. 
\end{enumerate}

\section*{Acknowledgements}
The authors thank Adam Van Tuyl and Kevin Vander Meulen for their helpful suggestions and insight. We also thank the anonymous referee for the quick and careful reading and giving helpful suggestions. A grant from the Infosys foundation partially supports Priyavrat Deshpande. Amit Roy thanks the Department of Atomic Energy, Government of India, for postdoctoral research fellowship. Anurag Singh is supported by the Start-up Research Grant SRG/2022/000314 from SERB, DST, India.

\bibliographystyle{abbrv}

\end{document}